\documentclass[11pt, reqno]{amsart}

\usepackage{graphicx}
\usepackage{epstopdf}
\usepackage{subcaption}
\usepackage[utf8]{inputenc}
\usepackage[T1]{fontenc}
\usepackage[english]{babel}
\usepackage{amsmath, amsthm}
\usepackage{ae}
\usepackage{icomma}
\usepackage{units}
\usepackage{color}
\usepackage{graphicx}
\usepackage{bbm}
\usepackage{caption}
\usepackage{array}
\usepackage[hmarginratio=1:1]{geometry}
\usepackage[hyphens]{url}
\usepackage[pdfpagelabels=false, hidelinks]{hyperref}

\usepackage{mathrsfs}
\usepackage{amssymb, amsfonts}
\usepackage{placeins} 
\usepackage{tikz}
\usetikzlibrary{patterns}
\usepackage{float}
\usepackage[nodayofweek]{datetime}
\usepackage{enumitem}
\usepackage{mathtools}
\usepackage[numbers, sort]{natbib}
\usepackage{dsfont}
\usepackage{csquotes}

\newcommand{\R}{\ensuremath{\mathbb{R}}}
\newcommand{\N}{\ensuremath{\mathbb{N}}}

\newcommand{\Z}{\ensuremath{\mathbb{Z}}}

\renewcommand{\S}{\ensuremath{\mathbb{S}}}
\newcommand{\Haus}{\ensuremath{\mathcal{H}}}

\newcommand{\1}{\ensuremath{\mathds{1}}}

\DeclareMathOperator{\Tr}{Tr}

\DeclareMathOperator{\supp}{supp}

\newtheorem{theorem}{Theorem}[section]

\newtheorem{lemma}[theorem]{Lemma}
\newtheorem{proposition}[theorem]{Proposition}

\numberwithin{theorem}{section}
\theoremstyle{remark}

\newcommand{\limplus}{{\mathchoice{\vcenter{\hbox{$\scriptstyle +$}}}
 {\vcenter{\hbox{$\scriptstyle +$}}}
 {\vcenter{\hbox{$\scriptscriptstyle +$}}}
 {\vcenter{\hbox{$\scriptscriptstyle +$}}}
}}
\newcommand{\limminus}{{\mathchoice{\vcenter{\hbox{$\scriptstyle -$}}}
 {\vcenter{\hbox{$\scriptstyle -$}}}
 {\vcenter{\hbox{$\scriptscriptstyle -$}}}
 {\vcenter{\hbox{$\scriptscriptstyle -$}}}
}}
\newcommand{\limpm}{{\mathchoice{\vcenter{\hbox{$\scriptstyle \pm$}}}
 {\vcenter{\hbox{$\scriptstyle \pm$}}}
 {\vcenter{\hbox{$\scriptscriptstyle \pm$}}}
 {\vcenter{\hbox{$\scriptscriptstyle \pm$}}}
}}

\begin{document}

\title[Robin counting functions in the critical scaling regime]{Robin counting functions on cuboids\\ in the critical scaling regime}

\author{Matthias Baur}
\address{\textnormal{(M. Baur)} Institute of Analysis, Dynamics and Modeling, Department of Mathematics, University of
Stuttgart, Pfaffenwaldring 57, 70569 Stuttgart, Germany.}
\email{\href{mailto:matthias.baur@mathematik.uni-stuttgart.de}{matthias.baur@mathematik.uni-stuttgart.de}}

\author{Simon Larson}
\address{\textnormal{(S. Larson)} Mathematical Sciences, Chalmers University of Technology and the University of Gothenburg, SE-412 96 G\"{o}teborg, Sweden. }
\email{\href{mailto:larsons@chalmers.se}{larsons@chalmers.se}}

\begin{abstract}
We consider eigenvalue counting functions of Robin Laplace operators on cuboids where the Robin parameter and the spectral cut-off $\lambda$ are coupled. Our main focus is the critical regime, in which the Robin parameter is proportional to $\sqrt{\lambda}$. We obtain a two-term asymptotic expansion for the counting functions in this coupled setting. In the critical regime, the second term in the asymptotic expansion depends non-trivially on the proportionality constant and interpolates continuously between the corresponding second terms for the Dirichlet and Neumann Laplacians. Outside the critical regime, one recovers the corresponding Dirichlet or Neumann asymptotics. We also establish a P\'olya-type inequality for the counting function whenever the ratio of the Robin parameter and $\sqrt{\lambda}$ exceeds a dimension-dependent threshold.
\end{abstract}

\maketitle


\section{Introduction and main results}

Let $\Omega \subset \R^d$ be a bounded open set with Lipschitz regular boundary. For $\beta\in \R$ we define the Robin Laplace operator
$-\Delta_\Omega^\beta$ on $L^2(\Omega)$ as the unique self-adjoint operator associated to the quadratic form 
$$
 u \mapsto\int_\Omega |\nabla u(x)|^2\, dx + \beta\int_{\partial \Omega} |u(x)|^2\, d\Haus^{d-1}(x)\, , \qquad u \in H^1(\Omega)\, , 
$$ 
see, e.g., \cite[Section~3.1]{Frank2022}. Since $\partial\Omega$ is Lipschitz, the quadratic form is lower semibounded and closed. The operator $-\Delta_\Omega^\beta$ has purely discrete spectrum, accumulating at infinity only. We denote its eigenvalues in non-decreasing order by $\{\lambda_k(-\Delta_\Omega^{\beta})\}_{k\geq 1}$, counting multiplicities, so that
\begin{equation*}
  \lambda_1(-\Delta_\Omega^\beta)\leq \lambda_2(-\Delta_\Omega^\beta)\leq \lambda_3(-\Delta_\Omega^\beta)\leq \ldots \to \infty\, .
\end{equation*}
We also introduce the Dirichlet and Neumann Laplace operators in $L^2(\Omega)$ and denote them by $-\Delta_\Omega^{\rm D}$ and $-\Delta_{\Omega}^{\rm N}$, respectively. The Neumann Laplace operator is obtained by taking $\beta = 0$ in the definition above while the Dirichlet Laplacian is defined through the quadratic form 
$$
 u \mapsto\int_\Omega |\nabla u(x)|^2\, dx\, , \qquad u \in H_0^1(\Omega)\, . 
$$

\subsection{Main results}

In this paper, we focus on the case $\beta>0$ and $\Omega \subset \mathbb{R}^d$, $d\geq 2$, belonging to the class of cuboids, i.e.\ sets of the form $\prod_{i=1}^d (0, l_i)$ with $l_1, \ldots, l_d >0$. We establish a two-term asymptotic expansion and a P\'olya-type inequality for the eigenvalue counting functions of Robin Laplace operators where the spectral cut-off $\lambda$ and the Robin parameter $\beta$ are coupled. Our main interest is in the critical scaling regime where $\beta \sim \sqrt{\lambda}$. Note that this regime is defined so that the length scale $\sim 1/\beta$ induced by the Robin condition remains comparable to the wavelength $\sim 1/\sqrt{\lambda}$ as $\lambda \to \infty$. Joint asymptotic regimes involving both the Robin parameter and the spectral parameter have recently been studied in several different settings. Asymptotics of Riesz means in the critical regime considered here were established in \cite{Frank2012, Baur2026}. Other coupled asymptotic regimes have been studied for negative Robin parameters; see, for instance, \cite{Kachmar2016, Helffer2017, Khalie2018, Helffer2022}.

To state our results, let us introduce some notation. For $d\geq 0, \gamma \geq 0$ we define the semiclassical constant by
\begin{equation*}
  L_{\gamma, d}^{\rm sc} := \frac{\Gamma(1+\gamma)}{(4\pi)^{d/2}\Gamma(1+\gamma+d/2)}\, , 
\end{equation*}
and for $\beta >0$ let
\begin{equation*}
  L_{\gamma, d}(\beta) := L_{\gamma, d}^{\rm sc}\biggl(\frac{4}{\pi}\int_0^1(1-s^2)^{\gamma+d/2}\frac{\beta}{\beta^2+s^2}\, ds -1\biggr)\, .
\end{equation*}
As is shown in \cite[Lemma 2.8]{Baur2026}, $\beta \mapsto L_{\gamma, d}(\beta)$ is continuous and strictly decreasing with
\begin{equation}\label{eq: limits of coefficient}
  \lim_{\beta \to 0^+}L_{\gamma, d}(\beta) = L_{\gamma, d}^{\rm sc} \quad \mbox{and} \quad \lim_{\beta \to \infty}L_{\gamma, d}(\beta) = -L_{\gamma, d}^{\rm sc}\, .
\end{equation}
For a self-adjoint operator $H$ whose spectrum consists of eigenvalues $\{\lambda_k(H)\}_{k\geq 1}$ we define the counting function by
\begin{align*}
  N(H, \lambda) := \#\{k:\lambda_k(H)\leq \lambda\}
\end{align*}
and for $\gamma >0, \lambda\in \R$ the Riesz mean of order $\gamma$ by
\begin{equation*}
  \Tr(H-\lambda)_\limminus^\gamma := \sum_{k\geq 1}(\lambda-\lambda_k(H))_\limplus^\gamma\, .
\end{equation*}
Here, $x_\limpm := (|x|\pm x)/2$. 

The first result of this paper is the following theorem which provides a two-term asymptotic expansion for the counting functions of Robin Laplacians in the critical scaling regime with explicit control of the remainder.

\begin{theorem}\label{thm: sc two-term Weyl intro}
  Fix $d\geq 2, \beta_0>0$. There exists a constant $C>0$ so that if $\beta \geq \beta_0$, $R\subset \R^d$ is a cuboid with side lengths $l_1, \ldots, l_d>0$, and $\lambda>0$ then
  \begin{align*}
    \Bigl|N(-\Delta_R^{\beta \sqrt{\lambda}}, \lambda) &- L_{0, d}^{\rm sc}|R|\lambda^{d/2}- \frac{1}{4}L_{0, d-1}(\beta)\Haus^{d-1}(\partial R)\lambda^{(d-1)/2}\Bigr|\\
    & \leq C \Haus^{d-1}(\partial R)\lambda^{(d-1)/2}\bigl((\min_i l_i \sqrt{\lambda})^{(1-d)/(d+1)}+ (\min_i l_i \sqrt{\lambda})^{1-d}\bigr)\, .
  \end{align*}
\end{theorem}
To our knowledge, Theorem \ref{thm: sc two-term Weyl intro} is the first two-term asymptotic expansion for counting functions of Robin Laplace operators in the critical scaling regime. Corresponding results for Riesz means of order $\gamma=1$ were obtained in \cite{Frank2012} on domains with $C^1$-boundary and more general Robin boundary conditions, and in \cite{Baur2026} for Riesz means of arbitrary positive order on cuboids.

Theorem~\ref{thm: sc two-term Weyl intro} also determines the asymptotic behavior outside the critical scaling regime. Specifically, if $\beta(\lambda)$ is a non-negative function satisfying $\lim_{\lambda \to \infty}\beta(\lambda)/\sqrt{\lambda} \in \{0, \infty\}$, then for any cuboid $R\subset \R^d$, 
\begin{equation}\label{eq: two-term asymptotics non-critical}
  N(-\Delta_R^{\beta(\lambda)}, \lambda) = L_{0, d}^{\rm sc}|R|\lambda^{d/2} \pm \frac{1}{4}L_{0, d-1}^{\rm sc}\Haus^{d-1}(\partial R)\lambda^{(d-1)/2} + o(\lambda^{(d-1)/2})\, , 
\end{equation}
where the second term comes with a plus if the limit of $\beta (\lambda)/\sqrt{\lambda}$ is zero and a minus if the limit is infinite. These asymptotics coincide with those that are well-known to hold for the Neumann and Dirichlet Laplace operators on $R$, respectively.

To see this, suppose first that $\lim_{\lambda \to \infty}\beta(\lambda)/\sqrt{\lambda} =0$. By monotonicity with respect to the Robin parameter, 
\begin{equation}\label{eq: comparison non-critical}
  \frac{N(-\Delta_R^{\beta' \sqrt{\lambda}}, \lambda)-L_{0, d}^{\rm sc}|R|\lambda^{d/2}}{\Haus^{d-1}(\partial R)\lambda^{(d-1)/2}}\leq \frac{N(-\Delta_R^{\beta(\lambda)}, \lambda)-L_{0, d}^{\rm sc}|R|\lambda^{d/2}}{\Haus^{d-1}(\partial R)\lambda^{(d-1)/2}}\leq \frac{N(-\Delta_R^{\rm N}, \lambda)-L_{0, d}^{\rm sc}|R|\lambda^{d/2}}{\Haus^{d-1}(\partial R)\lambda^{(d-1)/2}}\, , 
\end{equation}
for any $\beta'>0$ and $\lambda$ so large that $\beta'\geq \beta(\lambda)/\sqrt{\lambda}$. As $\lambda \to\infty$, the left-hand side of \eqref{eq: comparison non-critical} converges to $\frac{1}{4}L_{0, d-1}(\beta')$ by Theorem \ref{thm: sc two-term Weyl intro} while the right-hand side converges to $\frac{1}{4}L_{0, d-1}^{\rm sc}$ by the well-known asymptotics for $N(-\Delta_R^{\rm N}, \lambda)$. The asymptotics in \eqref{eq: two-term asymptotics non-critical} follow by sending $\beta' \to 0$ and appealing to \eqref{eq: limits of coefficient}. The case $\beta(\lambda)/\sqrt{\lambda}\to \infty$ follows analogously by comparing instead with $N(-\Delta_R^{\rm D}, \lambda)$ and $N(-\Delta_R^{\beta'\sqrt{\lambda}}, \lambda)$, and then letting $\beta'\to \infty$.

In particular, \eqref{eq: two-term asymptotics non-critical} recovers the classical asymptotics for fixed Robin Laplace operators. More generally, it shows that the second asymptotic coefficient depends on the Robin parameter only in the critical scaling regime.

Spectral asymptotics for fixed Laplace operators with Dirichlet, Neumann, or Robin boundary conditions have a long history. Weyl \cite{Weyl1911} established the leading-order asymptotics in rather general domains and conjectured that the two-term asymptotic formula should hold. For sets with smooth boundary satisfying an additional assumption on the associated billiard flow, Ivrii \cite{Ivrii1980} famously proved the validity of the conjectured two-term asymptotic formula. For Riesz means of arbitrary positive order, two-term asymptotics have recently been proved in bounded Lipschitz sets. For the Dirichlet and Neumann cases this was done in \cite{Frank2026a}, while the Robin case was treated in \cite{Frank2025}. 

For cuboids, analyzing the spectrum of the Dirichlet and Neumann Laplace operators is closely tied to lattice point counting, and the corresponding two-term asymptotic expansions can be deduced from classical lattice-point counting results. Although the connection is less direct in the Robin setting, it also forms the starting point of our proof of Theorem~\ref{thm: sc two-term Weyl intro}.

Given the results discussed above, it is natural to conjecture that the two-term asymptotics established here, and the corresponding asymptotics for Riesz means, extend to more general classes of domains and more general Robin boundary conditions.

The second main result of this paper concerns a Robin analogue of P\'olya's conjecture. Famously, P\'olya conjectured that the inequalities
\begin{equation*}
  N(-\Delta_\Omega^{\rm D}, \lambda) \leq L_{0, d}^{\rm sc}|\Omega|\lambda^{d/2}\leq N(-\Delta_\Omega^{\rm N}, \lambda)
\end{equation*}
hold for all $\lambda \geq 0$ and open $\Omega \subset \R^d$ with $|\Omega|<\infty$. Despite receiving considerable attention, P\'olya's conjecture remains open in general. In the special case of cuboids or, more generally, tiling domains, P{\'o}lya himself proved the conjecture~\cite{Polya1961}. Beyond these cases, only rather few examples of domains for which the conjecture is valid are known. Notably, P\'olya's conjecture was only very recently established when $\Omega$ is a ball \cite{Filonov2023, Filonov2026, Li2026}. The next theorem shows that among Robin Laplace operators on cuboids, the inequality conjectured by P\'olya for the Dirichlet case persists provided that the Robin parameter is sufficiently close to the Dirichlet endpoint.

\begin{theorem} \label{thm: sc_inequality}
Fix $d \geq 2$. Then there exists a constant $\beta(d)>0$ such that for any cuboid $R \subset \mathbb{R}^d$, $\beta \geq \beta(d)$, and $\lambda > 0$ we have
\begin{align*}
  N(-\Delta_{R}^{\beta\sqrt{\lambda}}, \, \lambda) \leq L_{0, d}^{\rm sc} |R| \lambda^{d/2}\, .
\end{align*}
Moreover, if $\beta > \beta(d)$ then the inequality is strict for each cuboid $R$ and $\lambda>0$. If instead $\beta <\beta(d)$ then there exists a cuboid $R$ and a $\lambda >0$ for which the inequality is violated. 
\end{theorem}

As noted already in \cite[Remark 5.2]{Baur2026}, for $d=1$ the bound in Theorem \ref{thm: sc_inequality} fails to hold for all $\lambda> 0$ independently of how large $\beta$ is chosen. The asymptotics of Theorem \ref{thm: sc two-term Weyl intro} and the monotonicity of $\beta \mapsto L_{0, d-1}(\beta)$ show that $\beta(d)$ must be at least as large as the solution of $L_{0, d-1}(\beta)=0$. We believe, however, that this lower bound for $\beta(d)$ is not attained, cf. \cite[Section 7]{Baur2026}.

 It is natural to ask whether the conclusion of Theorem~\ref{thm: sc_inequality} extends beyond cuboids. In analogy with P\'olya's conjecture, one may ask whether the inequality remains valid for all open sets of finite measure. However, since this would strengthen P\'olya's conjecture for the Dirichlet Laplacian, this question appears to be well beyond current techniques.

\subsection{Consequences for Riesz means}

Theorems \ref{thm: sc two-term Weyl intro} and \ref{thm: sc_inequality} can be applied to recover the corresponding results for Riesz means established in~\cite{Baur2026}. The key ingredient is the Aizenman--Lieb identity \cite{Aizenman1978}, in the form 
\begin{equation}\label{eq: AL identity}
  \Tr(-\Delta_R^{\beta\sqrt{\lambda}}-\lambda)_\limminus^\gamma=
\gamma
\int_0^\lambda
(\lambda-\mu)^{\gamma-1}
N(-\Delta_R^{\beta\sqrt{\lambda}}, \mu)
\, d\mu\, .
\end{equation}

From \eqref{eq: AL identity} combined with Theorem \ref{thm: sc two-term Weyl intro} it follows that for any $\gamma>0, \beta>0$, and cuboid $R\subset \R^d$, 
\begin{equation*}
    \Tr(-\Delta_{R}^{\beta\sqrt{\lambda}}-\lambda)_\limminus^\gamma = 
    L_{\gamma, d}^{\rm sc}|R|\lambda^{\gamma+d/2}+\frac{1}{4}L_{\gamma, d-1}(\beta) \Haus^{d-1}(\partial R)\lambda^{\gamma+(d-1)/2}+ o(\lambda^{\gamma+(d-1)/2})
\end{equation*}
as $\lambda \to \infty$, thereby recovering \cite[Theorem 1.2]{Baur2026} by a different argument. 

Similarly, combining \eqref{eq: AL identity} with Theorem \ref{thm: sc_inequality} yields that for any $\gamma>0$, $\beta \geq\beta(d)$, $\lambda \geq 0, $ and cuboid $R\subset \R^d$ it holds that
\begin{equation}\label{eq: Robin Berezin}
    \Tr(-\Delta_{R}^{\beta\sqrt{\lambda}}-\lambda)_\limminus^\gamma \leq L_{\gamma, d}^{\rm sc}|R|\lambda^{\gamma+d/2}\, .
\end{equation}
This inequality is of the same form as that obtained in \cite[Theorem 1.3]{Baur2026}. However, there it is proved that there exists a sharp threshold $\beta(\gamma, d)$ so that \eqref{eq: Robin Berezin} holds whenever $\beta \geq \beta(\gamma, d)$. We expect that the sharp threshold $\beta(\gamma,d)$ is strictly smaller than $\beta(d)$. What we denote in this article by $\beta(d)$ corresponds to $\beta(0, d)$ in the notation of \cite{Baur2026}. With the convention that $\beta(0, 1)=\infty$, the inequalities between these numbers in \cite[Corollary 5.5]{Baur2026} as well as the proofs given there extend to the case $\gamma=0$ as we now know that $\beta(0,d)$ is well-defined when $d\geq 2$. 

That one can obtain asymptotics and bounds for Riesz means from those of counting functions by integration is well-known. One point that requires some care here is that the counting functions appearing in \eqref{eq: AL identity} depend simultaneously on the parameters $\lambda$ and~$\mu$. This does not cause difficulties, however, since the error estimate in Theorem \ref{thm: sc two-term Weyl intro} is uniform with respect to $\beta \geq \beta_0$ and Theorem~\ref{thm: sc_inequality} is valid for all $\beta \geq \beta(d)$. Indeed, as $N(-\Delta_R^{\beta\sqrt{\lambda}}, \mu)= N(-\Delta_R^{\beta\sqrt{\lambda/\mu}\sqrt{\mu}}, \mu)$ and $\sqrt{\lambda/\mu}\beta \geq \beta$ for all $0< \mu\leq \lambda$, Theorem \ref{thm: sc two-term Weyl intro} and Theorem \ref{thm: sc_inequality} can therefore be applied to control the counting functions in the integral. For details of the computations for the two-term asymptotics we refer the reader to the proof of \cite[Theorem 2.4]{Baur2026}. There, the corresponding lifting is used to deduce asymptotics of Riesz means of order $\gamma>1$ from their validity at $\gamma=1$. The computation required here is analogous. In fact, the argument yields a non-asymptotic bound for the difference of the Riesz mean and its two-term asymptotic expansion analogous to that for the counting function in Theorem~\ref{thm: sc two-term Weyl intro}, and hence a bound of the same form as~\cite[Corollary 3.2]{Baur2026}. The error estimate obtained from lifting Theorem~\ref{thm: sc two-term Weyl intro} has a better asymptotic order than that in \cite[Corollary 3.2]{Baur2026}, except when $d=2$ and $\gamma\geq 1/2$. However, the result in \cite{Baur2026} has the advantage of providing a bound for the remainder that is completely independent of the parameter $\beta \in (0, \infty)$.

\subsection{An application in asymptotic spectral shape optimization}
\label{sec: shape opt}

The motivation for the present work comes from an asymptotic shape optimization problem in spectral theory. Specifically, we are interested in understanding how sets $\Omega$ which approximately realize the supremum
\begin{equation}
  \sup\bigl\{N(-\Delta_\Omega^{\beta(\lambda)}, \lambda): \Omega \subset \R^d \mbox{ open, } |\Omega|=1\bigr\} \label{eq:so_problem}
\end{equation}
behave as $\lambda \to \infty$ depending on the asymptotic behavior of $\beta(\lambda)$. While we are far from understanding this question in the stated generality, the results obtained in this paper allow us to give a rather satisfactory description when the optimization is restricted to the class of cuboids.

Define
\begin{equation*}
  M_d(\lambda, \beta) = \sup\bigl\{N(-\Delta_R^{\beta}, \lambda): R \subset \R^d \mbox{ a cuboid, } |R|=1\bigr\}\, , 
\end{equation*}
and recall that $\beta(\gamma, d)$ denotes the sharp threshold so that \eqref{eq: Robin Berezin} holds for all cuboids $R\subset \R^d$, $\lambda\geq 0$, and $\beta \geq \beta(\gamma, d)$. As a consequence of Theorem~\ref{thm: sc two-term Weyl intro}, we obtain the following description of sequences of cuboids which asymptotically realize the supremum defining $M_d(\lambda, \beta)$ as $\lambda \to \infty$.
\begin{proposition}\label{prop: shape opt}
  Fix $d\geq 2$. Let $\{\lambda_j\}_{j\geq1}, \{\beta_j\}_{j\geq 1}$ be sequences of positive numbers, and $\{R_j\}_{j\geq 1}$ be a sequence of cuboids in $\R^d$. Assume that $|R_j|=1$ for each $j\geq 1$, 
  \begin{equation*}
    \lim_{j\to \infty} \lambda_j = \infty\, , \quad \mbox{and}\quad \lim_{j\to \infty} \frac{N(-\Delta_{R_j}^{\beta_j}, \lambda_j)-M_d(\lambda_j, \beta_j)}{\lambda_j^{(d-1)/2}}=0\, .
  \end{equation*}
  \begin{enumerate}[label=(\roman*)]
    \item If $\limsup_{j\to \infty}\beta_j/\sqrt{\lambda_j}<\beta(1/2, d-1)$, then the sequence $\{R_j\}_{j\geq 1}$ has no converging subsequences.
    \item If $\liminf_{j\to \infty}\beta_j/\sqrt{\lambda_j}>\beta(1/2, d-1)$, then the sequence $\{R_j\}_{j\geq 1}$ converges to the unit cube as $j\to \infty$.
  \end{enumerate}
\end{proposition}

The analogue of Proposition \ref{prop: shape opt} in the setting of Riesz means of order $\gamma>0$ was obtained in \cite[Theorem 6.4]{Baur2026}, but the case $\gamma=0$ remained open there. To prove Proposition~\ref{prop: shape opt} one can follow the proof of \cite[Theorem 6.4]{Baur2026} verbatim except for replacing each application of \cite[Theorem 3.1]{Baur2026} by our Theorem~\ref{thm: sc two-term Weyl intro}.

As discussed in \cite{Baur2026}, we believe that $\beta(1/2, d-1)$ is strictly greater than the unique zero of the decreasing function $\beta \mapsto L_{0, d-1}(\beta)$. Thus, the behavioral transition point in Proposition~\ref{prop: shape opt} is not expected to coincide with the point at which the second term in the two-term asymptotic expansion for the counting function changes sign. This is noteworthy, since it shows that the asymptotics for a fixed cuboid, combined with the isoperimetric inequality, do not by themselves predict the transition in the behavior of optimizing sequences.

Asymptotic spectral shape optimization problems akin to \eqref{eq:so_problem} were first studied in \cite{Antunes2013a} where optimization for Dirichlet eigenvalues on rectangles was considered. Subsequently, the corresponding problems for Dirichlet and Neumann eigenvalues on cuboids in arbitrary dimension were treated in~\cite{vdBerg2016, vdBerg2017, Gittins2017}. Problems of this type have also been studied for Riesz means of Dirichlet and Neumann Laplace operators in~\cite{Larson_JST19, Frank2020, Frank2026b}. A common feature of these results is that, at least in certain regimes, the heuristic based on two-term asymptotics and the isoperimetric inequality can be made rigorous. For Robin eigenvalues on rectangles with a fixed Robin parameter, optimizing sequences have instead been shown to degenerate \cite{Antunes2013b, Freitas2019}. Degeneration of optimizing sequences also occurs in the analogous optimization problems for flat tori considered in \cite{Lagace2020}. Proposition \ref{prop: shape opt} exhibits both types of behavior, degeneration and convergence, with the transition between them occurring at a specific point in the critical scaling regime.

\subsection{Outline of proof strategies}
\label{sec: outline}

Both of our main results rely on reducing the spectral problem on a cuboid to suitable lattice-point counting problems. For the proof of Theorem~\ref{thm: sc two-term Weyl intro}, separation of variables together with two-sided estimates for the eigenvalues of the one-dimensional Robin Laplacian allows us to bound the Robin counting function from above and below by two lattice-point counts. After rescaling, these become counts in anisotropic dilations of convex sets which are small perturbations of a ball. A similar situation arises in the analysis of Steklov eigenvalues on cuboids in \cite{Girouard2019}, where the spectral problem is analyzed through approximating by lattice-point counts in a parameter-dependent family of convex sets.

To analyze the lattice-point problems arising, we establish in Section~\ref{sec: Lattice point theory} an estimate which is uniform over families of convex sets satisfying suitable uniform smoothness and curvature assumptions and which, in addition, keeps track of the effect of anisotropic dilations. The proof follows standard arguments based on Poisson summation and decay estimates for Fourier transforms of characteristic functions of convex sets; see, for instance, \cite[Section 7.7]{Hoermander2003}. The uniformity in the resulting estimate is essential for obtaining explicit control of the remainder in Theorem~\ref{thm: sc two-term Weyl intro} in terms of the shortest side length of the cuboid. This is particularly important for the application to shape optimization described in Section~\ref{sec: shape opt}, where the cuboid varies with the spectral parameter.

The proof of Theorem~\ref{thm: sc_inequality} also proceeds through lattice-point counting, but uses a different argument. We first dominate the Robin counting function by the number of points of a suitably shifted lattice lying in an ellipsoid, where the size of the shift tends to zero as the parameter $\beta$ tends to infinity. For sufficiently small shifts, this lattice count can be compared with the Dirichlet counting function at a slightly increased spectral cut-off. We then use a refined P\'olya-type bound for the Dirichlet counting function from~\cite{Gittins2017} to absorb the error introduced by increasing the cut-off. This yields the P\'olya-type inequality in Theorem~\ref{thm: sc_inequality} for all sufficiently large $\beta$.

\subsection*{Acknowledgment}

Financial support through the Swedish Research Council grant no.~2023-03985 (S.~L.) is acknowledged. M.~B.\ and S.~L.\ would both like to thank the Isaac Newton Institute for Mathematical Sciences, Cambridge, for support and hospitality during the programme `Geometric spectral theory and applications', where work on this paper was undertaken. The programme was supported by EPSRC grant EP/Z000580/1.

\section{On lattice point counts in convex sets}
\label{sec: Lattice point theory}

Counting lattice points in expanding convex sets is a classical problem with a long history. For general background on this topic, the reader is referred to~\cite{Gruber1987}. For sufficiently smooth convex sets with everywhere positive curvature, discrepancy estimates can be obtained by Fourier-analytic arguments based on the Poisson summation formula and the decay of the Fourier transform of the characteristic function of a convex set (see, for instance, \cite{Hlawka1950}). Considerably sharper discrepancy estimates have subsequently been obtained in various settings (see, for instance, \cite{Kraetzel1991, Kraetzel1992, Huxley2003}). 

For our purposes, the main issue is not obtaining the strongest possible discrepancy estimate, but rather obtaining sufficient uniformity in the underlying lattice-point problem. Uniform estimates under perturbations of the underlying convex set have previously been considered in \cite{Betke1999}, while lattice-point problems involving anisotropic dilations have appeared naturally in connection with spectral optimization \cite{Ariturk2017, Laugesen2018, Marshall2020, Lagace2020}. Particularly relevant to the present setting is \cite{Marshall2020}. Although the main results in \cite{Marshall2020} are formulated under additional restrictions, it is remarked there that the underlying lattice-point estimates remain valid without these restrictions and that they are uniform under perturbations of the underlying convex set. For completeness, we record below the corresponding estimate in the form required for our application and provide a proof following the classical Fourier-analytic approach.

To quantify the dependence on the anisotropic dilation, for a matrix $T\in \R^{d\times d}$ let $s_1(T)$ denote its smallest singular value, that is
\begin{equation*}
  s_1(T) = \min_{x \in \R^d\setminus \{0\}}\frac{|Tx|}{|x|}\, .
\end{equation*}
The estimate required for our application is the following.

\begin{proposition}\label{prop: Lattice point counting}
  Fix $d\geq 2$. There exists $m_d\in \N$ such that the following holds. Let $\mathcal{K}$ be a family of bounded convex subsets of $\R^d$ for which there exist constants $c_1, c_2, r_1, r_2>0$ such that:
  \begin{enumerate}
    \item the curvatures of $\partial E$ are bounded from below by $c_1$ for any $E \in \mathcal{K}$, 
    \item for any $E\in \mathcal{K}$ and $x \in \partial E$ the set $\partial E \cap B_{r_1}(x)$ can after rotation be parametrized as the graph of a function $g \in C^{m_d}(\R^{d-1})$ with
    \begin{equation*}
      \|g\|_{C^{m_d}(\R^{d-1})}\leq c_2\, , 
    \end{equation*}
    \item $B_{r_2}(0)\subset E$ for each $E \in \mathcal{K}$.
  \end{enumerate}

   Then there exists a constant $C_{\mathcal{K}}>0$ such that for all invertible matrices $T\in \R^{d\times d}$, $E \in \mathcal{K}$, and each $h \in (0, s_1(T)/2]$ it holds that
  \begin{equation*}
    \Bigl|h^d\#\{k \in (h \Z)^d \cap TE\} - |TE|\Bigr| \leq C_{\mathcal{K}}|TE| (h/s_1(T))^{2d/(d+1)}\, .
  \end{equation*}
\end{proposition}  

The fact that the estimate remains effective up to the regime $h \sim s_1(T)$ will be important in our proof of Theorem~\ref{thm: sc two-term Weyl intro}, where it allows us to control the remainder even when the shortest side length of the cuboid is comparable to $1/\sqrt{\lambda}$.

For the proof of Proposition 2.1, we follow the classical Fourier-analytic argument leading to \cite[Theorem 7.7.16]{Hoermander2003}. The key input is the following uniform Fourier decay estimate. 

\begin{lemma}[\cite{Hoermander2003}]\label{lem: Fourier decay}
  Fix $d\geq 2$. There exists $m_d\in \N$ such that the following holds. Let $\mathcal{K}$ be a family of bounded convex subsets of $\R^d$ for which there exist constants $c_1, c_2, r>0$ such that:
  \begin{enumerate}
    \item the curvatures of $\partial E$ are bounded from below by $c_1$ for any $E \in \mathcal{K}$, 
    \item for any $E\in \mathcal{K}$ and $x \in \partial E$ the set $\partial E \cap B_r(x)$ can after rotation be parametrized as the graph of a function $g \in C^{m_d}(\R^{d-1})$ with
    \begin{equation*}
      \|g\|_{C^{m_d}(\R^{d-1})}\leq c_2\, .
    \end{equation*}
  \end{enumerate}
  Then there exists a constant $C_{\mathcal{K}}>0$ such that 
  \begin{equation*}
    |\widehat{\1_E}(\xi)|\leq C_{\mathcal{K}} |\xi|^{-(d+1)/2}
  \end{equation*}
  for all $E\in \mathcal{K}$ and $\xi \in \R^d \setminus B_1(0)$.
\end{lemma}

Lemma \ref{lem: Fourier decay} follows from the proof of \cite[Lemma 7.7.15]{Hoermander2003}. The number of derivatives required for the estimate, and the dependence of the constant on these derivatives, are not explicitly tracked there. Tracking this dependence is merely a bookkeeping exercise, and for our application it suffices to know that uniform control of some finite number of derivatives is enough.

\begin{proof}[Proof of Proposition \ref{prop: Lattice point counting}]
  We follow closely the argument in \cite[Theorem 7.7.16]{Hoermander2003}.

  Fix a non-negative function $\varphi \in C_c^\infty(\R^d)$ with $\supp \varphi \subset B_{r_2/2}(0)$ and $\int_{\R^d}\varphi(x)\, dx=1$. For $\delta\in (0, 1)$, $E\in \mathcal{K}$, and invertible $T \in \R^{d\times d}$ we define
  \begin{equation*}
    \chi_{\delta, TE}(x) = |{\det(T)}|^{-1}\int_{\R^d} \varphi(T^{-1}y)\1_{TE}(x-\delta y)\, dy\, .
  \end{equation*}
  Since by assumption $\supp \varphi(T^{-1}\cdot) \subset TB_{r_2/2}(0)\subset (TE\cap (-TE))/2$ and $TE$ is convex it follows that
  \begin{equation*}
    \chi_{\delta, TE}((1-\delta)^{-1}x)\leq \1_{TE}(x) \leq \chi_{\delta, TE}((1+\delta)^{-1}x) \quad \mbox{for all }x\in \R^d\, .
  \end{equation*}
  In particular, for any $h>0$, 
  \begin{align*}
    \sum_{k \in (h\Z)^d} \chi_{\delta, TE}((1-\delta)^{-1}k) \leq \#\{k \in (h\Z)^d\cap TE\} \leq \sum_{k \in (h\Z)^d} \chi_{\delta, TE}((1+\delta)^{-1}k)\, .
  \end{align*}

  By the Poisson summation formula and basic properties of the Fourier transform, 
  \begin{align*}
    \sum_{k \in (h\Z)^d}& \chi_{\delta, TE}((1\pm \delta)^{-1}k)\\
    &=
    h^{-d}(1\pm \delta)^{d}\sum_{\xi \in \Z^d} \widehat{\chi_{\delta, TE}}((1\pm\delta)h^{-1}\xi)\\
    &=
    h^{-d}(1\pm \delta)^{d}|{\det(T)}|\sum_{\xi \in \Z^d} \widehat{\1_{E}}((1\pm\delta)h^{-1}T^{\top}\xi)\widehat{\varphi}((1\pm\delta)\delta h^{-1}T^{\top}\xi)\\
    &=
    h^{-d}(1\pm \delta)^{d}|TE|\\
    &\quad +
    h^{-d}(1\pm \delta)^{d}|{\det(T)}|\sum_{\xi \in \Z^d\setminus\{0\}} \widehat{\1_{E}}((1\pm\delta)h^{-1}T^{\top}\xi)\widehat{\varphi}((1\pm\delta)\delta h^{-1}T^{\top}\xi)\, .
  \end{align*}

  Lemma \ref{lem: Fourier decay} implies that if $|(1\pm \delta)h^{-1}T^{\top}\xi|\geq 1$ then 
  \begin{equation*}
    |\widehat{\1_{E}}((1\pm\delta)h^{-1}T^{\top}\xi)|\lesssim_{\mathcal{K}}(1\pm \delta)^{-(d+1)/2}h^{(d+1)/2}|T^{\top}\xi|^{-(d+1)/2}\, .
  \end{equation*}
  Similarly, the fact that $\varphi\in C_c^\infty(\R^d)$ implies that for any $N>0$
  \begin{equation*}
    |\widehat{\varphi}((1\pm\delta)\delta h^{-1}T^{\top}\xi)|\lesssim_N
     (1+(1\pm\delta)\delta h^{-1}|T^{\top}\xi|)^{-N}
  \end{equation*}
  for all $\xi \in \R^d$. Assuming that $\delta \leq 1/2$ and using the fact that $|T^\top \xi|\geq s_1(T^\top)|\xi|= s_1(T)|\xi|$ for all $\xi \in \R^d$ by the definition of the singular values, we conclude that
  \begin{align*}
    \biggl|\sum_{\xi \in \Z^d\setminus\{0\}}& \widehat{\1_{E}}((1\pm\delta)h^{-1}T^{\top}\xi)\widehat{\varphi}((1\pm\delta)\delta h^{-1}T^{\top}\xi)\biggr|\\
    &
    \lesssim_{d, N, \mathcal{K}}
     \delta^{(d+1)/2}\sum_{\xi \in \Z^d\setminus\{0\}} (s_1(T)\delta h^{-1}|\xi|)^{-(d+1)/2}(1+s_1(T)\delta h^{-1}|\xi|)^{-N}\, , 
  \end{align*}
  for all $h \in (0, s_1(T)/2]$.
  We bound the remaining sum by splitting it according to the size of $s_1(T)\delta h^{-1}|\xi|$ and comparing the resulting sums to the corresponding integrals
  \begin{align*}
    \sum_{\xi \in \Z^d\setminus\{0\}} &(s_1(T)\delta h^{-1}|\xi|)^{-(d+1)/2}(1+s_1(T)\delta h^{-1}|\xi|)^{-N}\\
    &=
    \sum_{\xi \in (\Z^d \cap B_{h/(s_1(T)\delta)}(0))\setminus\{0\}} (s_1(T)\delta h^{-1}|\xi|)^{-(d+1)/2}(1+s_1(T)\delta h^{-1}|\xi|)^{-N}\\
    &\quad +
    \sum_{\xi \in \Z^d \setminus B_{h/(s_1(T)\delta)}(0)} (s_1(T)\delta h^{-1}|\xi|)^{-(d+1)/2}(1+s_1(T)\delta h^{-1}|\xi|)^{-N}\\
    &\leq
    s_1(T)^{-(d+1)/2}\delta^{-(d+1)/2}h^{(d+1)/2}\sum_{\xi \in \Z^d \cap B_{h/(s_1(T)\delta)}(0)\setminus\{0\}} |\xi|^{-(d+1)/2}\\
    &\quad +
    s_1(T)^{-(d+1)/2-N}\delta^{-(d+1)/2-N}h^{(d+1)/2+N}\sum_{\xi \in \Z^d \setminus B_{h/(s_1(T)\delta)}(0)} |\xi|^{-(d+1)/2-N}\\
    &\lesssim_{d, N}
    s_1(T)^{-d}\delta^{-d}h^{d}\, , 
  \end{align*}
  assuming that $N>(d-1)/2$. We may now choose $N= (d+1)/2$.

  Gathering the estimates above we have shown that for all $\delta \in (0, 1/2], h\in (0, s_1(T)/2]$, each $E \in \mathcal{K}$, and invertible $T \in \R^{d\times d}$
  \begin{align*}
    \Bigl|h^d\#\{k \in (h\Z)^d\cap TE\} - |TE|\Bigr|
    &\lesssim_{d, \mathcal{K}}
    ((1+\delta)^d-(1-\delta)^d)|{\det (T)}||E|\\
    &\quad + |{\det(T)}| s_1(T)^{-d}h^d\delta^{-(d-1)/2}\\
    &\lesssim_{d, \mathcal{K}}
    \delta|TE|\\
    &\quad + |{\det(T)}| s_1(T)^{-d}h^d\delta^{-(d-1)/2}\, .
  \end{align*}
  Choosing $\delta=(h/s_1(T))^{2d/(d+1)}$, which is less than $1/2$ for all $h\leq s_1(T)/2$ since $2d/(d+1)\geq 1$, and using the facts that $|TE|=|\det(T)||E|$ and $|E|\gtrsim_\mathcal{K} 1$ for all $E\in \mathcal{K}$ by the assumptions, we conclude that
  \begin{align*}
    \Bigl|h^d\#\{k \in (h\Z)^d\cap TE\} - |TE|\Bigr|
    &\lesssim_{d, \mathcal{K}}
    |TE|(h/s_1(T))^{2d/(d+1)}\, , 
  \end{align*}
  which completes the proof of Proposition \ref{prop: Lattice point counting}.
\end{proof}

\section{Asymptotics for the Robin counting function}

The goal of this section is to prove Theorem \ref{thm: sc two-term Weyl intro}.

\subsection{Estimating the spectral counting function by lattice point counts}

The first step of the proof is based on the following two-sided bounds for eigenvalues of the Robin Laplacian on an interval.

\begin{lemma} \label{lem: arctan bounds lambdak}
  For each $k\geq 1$ and any $\beta>0$, 
  \begin{equation*}
    \Bigl(\pi k -2\arctan\Bigl(\frac{\pi k}{\beta}\Bigr)\Bigr)^2 < \lambda_k(-\Delta_{(0, 1)}^\beta)< \Bigl(\pi k- 2\arctan\Bigr(\frac{\pi k}{\beta}- \frac{2}{\beta}\arctan\Bigl(\frac{\pi k}{\beta}\Bigr)\Bigr)\Bigr)^2\, .
  \end{equation*}
\end{lemma}

\begin{proof}
  Define $\delta_k(\beta):=\pi k - \sqrt{\lambda_k(-\Delta_{(0, 1)}^\beta)}\in (0, \pi)$. As argued in the proof of \cite[Lemma 2.3]{Baur2026}, we have $\delta_k(\beta) \in (0, \pi)$ for any $k$ and any $\beta>0$ and $\delta_k(\beta)$ satisfies the implicit equation 
  \begin{align*}
    \delta_k(\beta) = 2\arctan\Bigl(\frac{\pi k-\delta_k(\beta)}{\beta}\Bigr)\, .
  \end{align*}
  By monotonicity of $x \mapsto \arctan(x)$ and iteration of the implicit equation, we obtain
  \begin{align*}
   2\arctan\Bigl(\frac{\pi k}{\beta} - \frac{2}{\beta }\arctan\Bigl(\frac{\pi k}{\beta}\Bigr)\Bigr) < \delta_k(\beta) < 2\arctan\Bigl(\frac{\pi k}{\beta}\Bigr)\, 
  \end{align*}
  from which the claimed inequalities follow.  
\end{proof}

For each $k\geq 1$ and $\beta>0$, the fact that $\arctan(y)\in (0, \pi/2)$ for $y>0$ implies that
\begin{equation*}
  \pi^2(k-1)^2<\Bigl(\pi k -2\arctan\Bigl(\frac{\pi k}{\beta}\Bigr)\Bigr)^2 \leq \Bigl(\pi k- 2\arctan\Bigr(\frac{\pi k}{\beta}- \frac{2}{\beta}\arctan\Bigl(\frac{\pi k}{\beta}\Bigr)\Bigr)\Bigr)^2 <\pi^2k^2\, .
\end{equation*}
In other words, the inequalities provided by Lemma~\ref{lem: arctan bounds lambdak} improve upon the well-known Dirichlet-Neumann bracket 
\begin{equation}\label{eq: DN bracketing 1D}
  \lambda_k(-\Delta_{(0, 1)}^{\rm N})<\lambda_k(-\Delta_{(0, 1)}^\beta)<\lambda_k(-\Delta_{(0, 1)}^{\rm D})\, .
\end{equation}

For $\beta>0, \lambda\geq 0$, and a cuboid $R\subset \R^d$ with side lengths $l_1, \ldots, l_d>0$, define
\begin{align*}
  N^-_{\beta, R}(\lambda) &= \#\Biggl\{k \in \N^d: \sum_{i=1}^d \frac{1}{l_i^2}\Bigl(\pi k_i -2\arctan\Bigl(\frac{\pi k_i}{\beta \sqrt{\lambda}l_i}\Bigr)\Bigr)^2\leq \lambda\Biggr\}\, , \\
  N^+_{\beta, R}(\lambda) &= \#\Biggl\{k \in \N^d: \sum_{i=1}^d \frac{1}{l_i^2}\Bigl(\pi k_i -2\arctan\Bigl(\frac{\pi k_i}{\beta \sqrt{\lambda}l_i}-\frac{2}{\beta \sqrt{\lambda}l_i}\arctan\Bigl(\frac{\pi k_i}{\beta \sqrt{\lambda}l_i}\Bigr)\Bigr)\Bigr)^2\leq \lambda\Biggr\}\, .
\end{align*}
Lemma \ref{lem: arctan bounds lambdak} combined with \eqref{eq: DN bracketing 1D}, separation of variables, and the behavior of the Robin Laplacian eigenvalues under scaling yields the bounds
\begin{equation}\label{eq: DN bracketing counting}
  N(-\Delta_R^{\rm D}, \lambda) \leq N^+_{\beta, R}(\lambda) \leq N(-\Delta_R^{\beta\sqrt{\lambda}}, \lambda)\leq N^-_{\beta, R}(\lambda) \leq N(-\Delta_R^{\rm N}, \lambda)
\end{equation}
for all $\beta> 0, \lambda \geq 0$, and cuboids $R\subset \R^d$. In particular, \eqref{eq: DN bracketing counting} shows that if we can prove matching two-term asymptotic expansions for $N_{\beta, R}^\pm(\lambda)$ then the Robin counting function must satisfy the same asymptotics. Furthermore, the comparison of $N^\pm_{\beta, R}$ with the Dirichlet and Neumann counting functions yields the following useful a priori bounds that capture the correct leading-order asymptotic behavior. 
\begin{lemma}\label{lem: first order asymptotitcs}
  Fix $d\geq 1$. There exists a constant $C_d>0$ so that for any $\beta>0, \lambda> 0$, and cuboid $R \subset \R^d$ with side lengths $l_1, \ldots, l_d>0$, 
  \begin{equation*}
    \Bigl|N^{\pm}_{\beta, R}(\lambda) - L_{0, d}^{\rm sc}|R|\lambda^{d/2}\Bigr| \leq C_d \Haus^{d-1}(\partial R) \lambda^{(d-1)/2}(1+(\min_i l_i \sqrt{\lambda})^{1-d})\, . 
  \end{equation*}
\end{lemma}
\begin{proof}
  The claimed bound follows directly from \eqref{eq: DN bracketing counting} combined with \cite[Lemma 3.4]{Baur2026}.
\end{proof}

\subsection{On the asymptotic behavior of \texorpdfstring{$N^\pm_{\beta, R}$}{the lattice points counting functions}}
Our aim in this section is to show that the asymptotic behavior of $N^\pm_{\beta, R}$ can be understood through the results in Section \ref{sec: Lattice point theory}. The first step is to make a change of variables so that the lattice point counts defining $N^\pm_{\beta, R}$ are phrased in the language used in Section \ref{sec: Lattice point theory}.

For a cuboid $R \subset \R^d$ with side lengths $l_1, \ldots, l_d>0$ we define the matrix $T_R= \mathrm{diag}(l_1, \ldots, l_d)$. Furthermore, we define for $x\in \R^d, \beta>0$, and $\delta_1, \ldots , \delta_d\geq 0$, the functions
\begin{align*}
  q_{\beta, \delta_1, \ldots, \delta_d}^-(x) &= \sum_{i=1}^d\Bigl(\pi x_i- 2\delta_i\arctan\Bigl(\frac{\pi x_i}{\beta}\Bigr)\Bigr)^2\\
  q_{\beta, \delta_1, \ldots, \delta_d}^+(x) &= \sum_{i=1}^d\Bigl(\pi x_i- 2\delta_i\arctan\Bigl(\frac{\pi x_i}{\beta}- \frac{2\delta_i}{\beta}\arctan\Bigr(\frac{\pi x_i}{\beta}\Bigr)\Bigr)\Bigr)^2\, .
\end{align*}
Note that the functions $q_{\beta, \delta_1, \ldots, \delta_d}^\pm$ are invariant under reflections in the coordinate hyperplanes.

With these definitions in hand we can express our approximate counting functions as
\begin{align*}
  N_{\beta, R}^\pm(h^{-2}) &= \# \{ k \in (h\mathbb{N})^d: q_{\beta, h/l_1, \ldots, h/l_d}^\pm(T_R^{-1}k) \leq 1 \}\, .
\end{align*}
In particular, if we define
\begin{equation*}
  E^\pm_{\beta, \delta_1, \ldots, \delta_d} = \{x\in \R^d: q^\pm_{\beta, \delta_1, \ldots, \delta_d}(x)\leq 1\}
\end{equation*}
then
\begin{equation*}
  N_{\beta, R}^\pm(h^{-2}) = \# \{ k \in (h\mathbb{N})^d \cap T_R E^\pm_{\beta, h/l_1, \ldots, h/l_d}\}\, .
\end{equation*}
Essentially, the geometry of the cuboid is encoded by the anisotropic dilation $T_R$, while the deviation of the Robin eigenvalues from the Dirichlet ones is encoded by how the set $E_{\beta, h/l_1, \ldots, h/l_d}^\pm$ differs from a ball.

We now apply Proposition~\ref{prop: Lattice point counting} to obtain precise asymptotics for these lattice point counting functions. We begin by proving the following lemma.

\begin{lemma}\label{lem: asymptotics in terms of Epm measure}
  Fix $d\geq 2, \beta_0>0$. There exist constants $\delta, C >0$ such that if $\beta \geq \beta_0$, $R\subset \R^d$ is a cuboid with side lengths $l_1, \ldots, l_d>0$, and $0<h \leq \delta \min_i l_i$ then
  \begin{align*}
    \biggl|N^\pm_{\beta, R}(h^{-2})-\frac{1}{2^d}|R||E^\pm_{\beta, h/l_1, \ldots, h/l_d}|h^{-d} + &\frac{1}{4}L_{0, d-1}^{\rm sc}\Haus^{d-1}(\partial R)h^{-d+1}\biggr| \\
    &\leq C |R|h^{-d}\Bigl(\frac{h}{\min_i l_i}\Bigr)^{2d/(d+1)}\, .
  \end{align*}
\end{lemma}

\begin{proof}
  Note that, in contrast to the counting functions featured in Proposition \ref{prop: Lattice point counting}, the functions $N^\pm_{\beta, R}$ count only lattice points with positive coordinates. However, since the sets $T_RE^\pm_{\beta, \delta_1, \ldots, \delta_d}$ are axis-symmetric this is not difficult to deal with. Indeed, the axis symmetry of these sets implies that the number of lattice points in the entire set can be related to the number in the positive orthant by compensating by appropriate counts on the coordinate hyperplanes. To this end, for $I \subsetneq \{1, \ldots, d\}$, let $R_I=\prod_{i \notin I}(0, l_i)$, i.e., $R_I$ is the $(d-|I|)$-dimensional cuboid obtained by removing the directions indexed by $I$. Then we can express the functions $N_{\beta, R}^\pm$ as 
  \begin{align}
  \begin{split}
  N_{\beta, R}^\pm(h^{-2})
    &=
    2^{-d}\#\{x \in (h\Z)^d \cap T_RE^\pm_{\beta, h/l_1, \ldots, h/l_d}\}\\
    &\quad -\sum_{n=1}^{d-1} 2^{-n}\sum_{1\leq i_1 <\ldots <i_n \leq d} N^\pm_{\beta, R_{\{i_1, \ldots, i_n\}}}(h^{-2})-2^{-d}\, . \label{eq:countingfunction_puzzled_together}
  \end{split}
  \end{align}
  
 Let us consider first the coordinate-hyperplane corrections in the second line of \eqref{eq:countingfunction_puzzled_together}. By Lemma~\ref{lem: first order asymptotitcs}, among the coordinate-hyperplane corrections, only the leading-order contributions from the terms with $n=1$ matter for our purposes. Indeed, 
  \begin{align*}
    \biggl|\sum_{n=1}^{d-1} 2^{-n}&\sum_{1\leq i_1 <\ldots <i_n \leq d} N^\pm_{\beta, R_{\{i_1, \ldots, i_n\}}}(h^{-2})+2^{-d}
    -\frac{1}{2} L_{0, d-1}^{\rm sc} h^{-d+1}\sum_{i=1}^d \frac{|R|}{l_i}\biggr|\\
    &\lesssim_{d}
    \sum_{i=1}^d\frac{|R|}{l_i \min_{j\neq i}l_j}h^{-d+2}+\sum_{n=2}^{d-1} 2^{-n}\sum_{1\leq i_1 <\ldots <i_n \leq d} \frac{|R|}{\prod_{j=1}^nl_{i_j}}h^{-d+n}+2^{-d}\\
    &\lesssim_{d} 
    |R|h^{-d}\Bigl(\frac{h}{\min_i l_i}\Bigr)^2
  \end{align*}
  for all cuboids $R, \beta>0, $ and $h \in (0, \min_i l_i]$. Since $\Haus^{d-1}(\partial R) = 2 \sum_{i=1}^d l_i^{-1}|R|$, we have shown that
  \begin{equation}\label{eq: bound coordinate plane terms}
  \begin{aligned}
    \biggl|N_{\beta, R}^\pm(h^{-2})
    -
    2^{-d}\#\{x \in (h\Z)^d \cap T_RE^\pm_{\beta, h/l_1, \ldots, h/l_d}\}+ &\frac{1}{4}L_{0, d-1}^{\rm sc}\Haus^{d-1}(\partial R)h^{-d+1}\biggr|\\
    &\lesssim_d |R|h^{-d}\Bigl(\frac{h}{\min_i l_i}\Bigr)^2\, , 
  \end{aligned}
  \end{equation}
  for all cuboids $R, \beta>0, $ and $h \in (0, \min_i l_i]$.

Now consider the lattice count $\#\{x \in (h\Z)^d \cap T_RE^\pm_{\beta, h/l_1, \ldots, h/l_d}\}$. We claim that for any $d \geq 2$ and $\beta_0>0$ there exists $\delta>0$ so that the collection of sets defined by
  \begin{equation*}
    \mathcal{K}= \bigl\{E^\pm_{\beta, \delta_1, \ldots, \delta_d}: \beta\geq \beta_0, \max_i \delta_i \leq \delta\bigr\}\, , 
  \end{equation*}
  satisfies the assumptions in Proposition \ref{prop: Lattice point counting}. Indeed, this follows from the fact that $q^\pm_{\beta, \delta_1, \ldots, \delta_d}$ and its derivatives of any fixed order converge uniformly on compact sets to $q_0(x) = \pi^2|x|^2$ and its corresponding derivatives as $\max_{i}\delta_i \to 0$, uniformly for $\beta \geq \beta_0$. In particular, for $\beta \geq \beta_0$ and $h\in (0, \delta\min_i l_i]$ it follows from Proposition \ref{prop: Lattice point counting} that
  \begin{equation}\label{eq: Epm lattice count}
    \Bigl|\#\{x \in (h\Z)^d \cap T_RE^\pm_{\beta, h/l_1, \ldots, h/l_d}\}- h^{-d}|R||E_{\beta, h/l_1, \ldots, h/l_d}^\pm|\Bigr|
    \lesssim_{d, \beta_0, \delta} |R| h^{-d} \Bigl(\frac{h}{\min_i l_i}\Bigr)^{2d/(d+1)}
  \end{equation}
  provided $\delta \leq 1/2$. Here we used that $\det(T_R)=|R|$, $s_1(T_R)= \min_i l_i$, and that the measures of the sets in the collection $\mathcal{K}$ are uniformly bounded.

  Putting together \eqref{eq: bound coordinate plane terms} and \eqref{eq: Epm lattice count} completes the proof of Lemma \ref{lem: asymptotics in terms of Epm measure}.
\end{proof}

To finalize our analysis of the asymptotic behavior of $N^{\pm}_{\beta, R}$ it remains to provide a precise asymptotic description of the measures $|E^\pm_{\beta, \delta_1, \ldots, \delta_d}|$ in the appropriate asymptotic regime. 

\begin{lemma}\label{lem: Epm measure}
  Fix $d\geq 2, \beta_0>0$. There exist $\delta>0, C>0$ such that for $\beta \geq \beta_0$ and $0\leq \delta_1, \ldots, \delta_d \leq \delta$
  we have
  \begin{equation*}
    \biggl||E^\pm_{\beta, \delta_1, \ldots, \delta_d}| - 2^dL_{0, d}^{\rm sc}- 2^{d-1}(L_{0, d-1}(\beta)+L_{0, d-1}^{\rm sc})\sum_{i=1}^d \delta_i \biggr| \leq C \max_i \delta_i^2\, .
  \end{equation*}
\end{lemma}

\begin{proof}
  We argued in the proof of Lemma~\ref{lem: asymptotics in terms of Epm measure} that, for $\delta$ sufficiently small, the sets $E^\pm_{\beta, \delta_1, \ldots, \delta_d}$ are smooth convex sets which can be viewed as small perturbations of the ball $B_{1/\pi}(0)$. In particular, their boundaries can be parametrized in spherical coordinates as follows. For $\theta \in \S^{d-1}$ we define $r^\pm_{\beta, \delta_1, \ldots, \delta_d}(\theta)$ as the unique positive number so that $r^\pm_{\beta, \delta_1, \ldots, \delta_d}(\theta)\theta \in \partial E^\pm_{\beta, \delta_1, \ldots, \delta_d}$.

  The measure of $E^\pm_{\beta, \delta_1, \ldots, \delta_d}$ can thus be computed in spherical coordinates as
  \begin{equation}\label{eq: spherical volume rep}
  \begin{aligned}
    |E^{\pm}_{\beta, \delta_1, \ldots, \delta_d}| &= \int_{\S^{d-1}}\int_0^{r^\pm_{\beta, \delta_1, \ldots, \delta_d}(\theta)} r^{d-1}\, dr d\Haus^{d-1}(\theta)\\
    &=
    \frac{1}{d}\int_{\S^{d-1}} r_{\beta, \delta_1, \ldots, \delta_d}^\pm (\theta)^d \, d\Haus^{d-1}(\theta)\, .
  \end{aligned}
  \end{equation}
  To obtain the desired asymptotic bound, it suffices to derive corresponding asymptotics for $r^\pm_{\beta, \delta_1, \ldots, \delta_d}(\theta)$.

  By the definition of $E_{\beta, \delta_1, \ldots, \delta_d}^\pm$ we have that
  \begin{equation}\label{eq: implicit eq r}
  \begin{aligned}
    1&= q_{\beta, \delta_1, \ldots, \delta_d}^\pm(r_{\beta, \delta_1, \ldots, \delta_d}^\pm(\theta)\theta)\\
    &=
    \pi^2 r_{\beta, \delta_1, \ldots, \delta_d}^\pm(\theta)^2-4\pi r_{\beta, \delta_1, \ldots, \delta_d}^\pm(\theta) \sum_{i=1}^d \delta_i \theta_i\arctan\Bigl(\frac{\pi y_i^\pm}{\beta}\Bigr) +4 \sum_{i=1}^d\delta_i^2 \arctan\Bigl(\frac{\pi y_i^\pm}{\beta}\Bigr)^2
  \end{aligned}
  \end{equation}
  where
  \begin{equation*}
    y_i^-= r_{\beta, \delta_1, \ldots, \delta_d}^-(\theta)\theta_i \quad \mbox{and}\quad
    y_i^+= r_{\beta, \delta_1, \ldots, \delta_d}^+(\theta)\theta_i- \frac{2\delta_i}{\pi}\arctan\Bigl(\frac{\pi r_{\beta, \delta_1, \ldots, \delta_d}^+(\theta)\theta_i}{\beta}\Bigr)\, .
  \end{equation*}
  
  Since $\arctan$ is bounded, it follows that
  \begin{equation*}
    |r_{\beta, \delta_1, \ldots, \delta_d}^\pm(\theta)-1/\pi|\lesssim_{d, \beta_0, \delta}\max_i \delta_i \quad \mbox{and}\quad \Bigl|y^\pm_i - \frac{\theta_i}{\pi}\Bigr|\lesssim_{d, \beta_0, \delta} \max_j \delta_j\, .
  \end{equation*}
  Consequently, since all derivatives of $\arctan$ are uniformly bounded, it follows from \eqref{eq: implicit eq r} and Taylor's theorem that
  \begin{equation}\label{eq: r expansion}
    \Bigl|r_{\beta, \delta_1, \ldots, \delta_d}^\pm(\theta)-\frac{1}\pi- \frac{2}{\pi} \sum_{i=1}^d \delta_i \theta_i\arctan\Bigl(\frac{\theta_i}{\beta}\Bigr)\Bigr|\lesssim_{d, \beta_0, \delta} \max_{i}\delta_i^2\, .
  \end{equation}

  Combining \eqref{eq: r expansion} with \eqref{eq: spherical volume rep}, we conclude that
  \begin{equation*}
    \biggl||E_{\beta, \delta_1, \ldots, \delta_d}^\pm| 
    - \frac{|B_{1}|}{\pi^d} - \frac{2}{\pi^d} \sum_{i=1}^d \delta_i \int_{\S^{d-1}}\theta_i \arctan\Bigl(\frac{\theta_i}{\beta}\Bigr)\, d\Haus^{d-1}(\theta)\biggr| \lesssim_{d, \beta_0, \delta} \max_i \delta_i^2\, .
  \end{equation*}

  Since $(2\pi)^{-d}|B_1|= L_{0, d}^{\rm sc}$, it remains to compute the resulting integrals. By symmetry the integral
  \begin{align*}
    \int_{\mathbb{S}^{d-1}}  \theta_i \arctan\Big( \frac{\theta_i}{\beta }\Big) \, d\Haus^{d-1}(\theta) 
  \end{align*}
  does not depend on $i$. Thus, without loss of generality, we can compute it for $i=1$. Using the explicit formula for $L_{0, d-1}^{\rm sc}$ we find that
  \begin{align*}
    \int_{\mathbb{S}^{d-1}}  \theta_1 \arctan\Bigl( \frac{\theta_1}{\beta }\Bigr) \, d\Haus^{d-1}(\theta) &= 2 |\mathbb{S}^{d-2}| \int_0^1 x(1-x^2)^{(d-3)/2} \arctan\Bigl(\frac{x}{\beta}\Bigr) \, dx \\
    &=\frac{4\pi^{(d-1)/2}}{\Gamma((d-1)/2)} \int_0^1 x(1-x^2)^{(d-3)/2} \arctan\Bigl(\frac{x}{\beta}\Bigr) \, dx \\
    &= 2^{d-2}\pi^{d} L_{0, d-1}^{sc} \biggl[ \frac{8}{\pi} \frac{d-1}{2} \int_0^1 x(1-x^2)^{(d-3)/2} \arctan\Bigl(\frac{x}{\beta}\Bigr) \, dx \biggr] \\
    &=2^{d-2}\pi^{d} (L_{0, d-1}(\beta) + L_{0, d-1}^{sc})\, , 
  \end{align*}
where in the final step, we used the identity
\begin{equation*}
    L_{\gamma, d}(\beta) = L_{\gamma, d}^{\rm sc}\biggl[\frac{8}{\pi}\Bigl(\gamma+\frac{d}{2}\Bigr)\int_0^1 x(1-x^2)^{\gamma+d/2-1}\arctan\Bigl(\frac{x}{\beta}\Bigr)\, dx - 1\biggr]\, 
  \end{equation*}
from \cite[Lemma 2.8]{Baur2026}. This completes the proof of Lemma \ref{lem: Epm measure}.
\end{proof}

Combining the preceding results yields the desired two-term estimate for $N^\pm_{\beta, R}$.

\begin{theorem}\label{thm: sc two-term Weyl Npm}
  Fix $d\geq 2, \beta_0>0$. There exists a constant $C>0$ so that if $\beta \geq \beta_0$, $R\subset \R^d$ is a cuboid with side lengths $l_1, \ldots, l_d>0$, and $\lambda>0$ then
  \begin{align*}
    \Bigl|N^\pm_{\beta, R}(\lambda) &- L_{0, d}^{\rm sc}|R|\lambda^{d/2}- \frac{1}{4}L_{0, d-1}(\beta)\Haus^{d-1}(\partial R)\lambda^{(d-1)/2}\Bigr|\\
    & \leq C \Haus^{d-1}(\partial R)\lambda^{(d-1)/2}\bigl((\min_i l_i \sqrt{\lambda})^{(1-d)/(d+1)}+ (\min_i l_i \sqrt{\lambda})^{1-d}\bigr)\, .
  \end{align*}
\end{theorem}

\begin{proof}
  By combining Lemmas~\ref{lem: asymptotics in terms of Epm measure} and \ref{lem: Epm measure}, we conclude that there exists $\delta>0$ such that
  \begin{align*}
    \biggl|N^\pm_{\beta, R}(\lambda)-&L_{0, d}^{\rm sc}|R|\lambda^{d/2}-\frac{1}{4}L_{0, d-1}(\beta)\Haus^{d-1}(\partial R)\lambda^{(d-1)/2}\biggr|\\
    &\lesssim_{d, \beta_0} \Haus^{d-1}(\partial R)\lambda^{(d-1)/2}(\min_i l_i \sqrt{\lambda})^{(1-d)/(d+1)}
  \end{align*}
  whenever $\beta \geq \beta_0$ and $\min_il_i\sqrt{\lambda} \geq \delta^{-1}$. Here we used again that $\Haus^{d-1}(\partial R) = 2\sum_{i=1}^d \frac{|R|}{l_i}$ and note that the terms from Lemmas~\ref{lem: asymptotics in terms of Epm measure} and \ref{lem: Epm measure} involving the constant $L_{0,d-1}^{sc}$ cancel.

  If instead $0<\min_il_i\sqrt{\lambda} \leq \delta^{-1}$, then the triangle inequality, Lemma~\ref{lem: first order asymptotitcs}, and the bound $|L_{0, d-1}(\beta)|\lesssim_d 1$ yield
  \begin{align*}
    \biggl|N^\pm_{\beta, R}(\lambda)-&L_{0, d}^{\rm sc}|R|\lambda^{d/2}-\frac{1}{4}L_{0, d-1}(\beta)\Haus^{d-1}(\partial R)\lambda^{(d-1)/2}\biggr|\\
    &\lesssim_{d, \delta}
    \Haus^{d-1}(\partial R) \lambda^{(d-1)/2}(\min_i l_i \sqrt{\lambda})^{1-d}\, .
  \end{align*}
  This completes the proof of the theorem.
\end{proof}

An immediate consequence of Theorem \ref{thm: sc two-term Weyl Npm} is the validity of Theorem~\ref{thm: sc two-term Weyl intro}.
\begin{proof}[Proof of Theorem~\ref{thm: sc two-term Weyl intro}]
  The claimed inequality follows by combining the inequalities in \eqref{eq: DN bracketing counting} with the bounds in Theorem~\ref{thm: sc two-term Weyl Npm}.
\end{proof}

\section{Uniform upper bound on the counting function}

In this section, we prove Theorem~\ref{thm: sc_inequality}. As in the case of Theorem~\ref{thm: sc two-term Weyl intro}, our argument is based on comparing the Robin counting function with a suitable lattice-point count. Here, however, a simpler comparison will suffice. The strategy is to prove that $N(-\Delta_R^{\beta \sqrt{\lambda}}, \lambda)$ is dominated by a lattice-point count in a shifted lattice, with the shift chosen in terms of $\beta$. We then show that, for all sufficiently small shifts, these lattice-point counts satisfy a corresponding P\'olya-type bound. More specifically, we shall consider the counting function
\begin{align*}
N_{s, R}(\lambda) = \# \{ k \in \mathbb{N}^d : Q_{s, R}(k) \leq \lambda \} 
\end{align*}
where $s\in [0, \pi]$, $R = \prod_{i=1}^d (0, l_i)$, and
\begin{align*}
Q_{s, R}(k) = \sum_{i=1}^d \frac{1}{l_i^2} (\pi k_i - s )^2 \,.
\end{align*}
In other words, $N_{s, R}(\lambda)$ counts the number of points of the lattice $\prod_{i=1}^dl_i^{-1}(\pi \Z-s)$ that have positive, or non-negative if $s=\pi$, coordinates and lie in the ball $\overline{B_{\sqrt{\lambda}}(0)}$. Note that $N_{s, R}(\lambda)$ is non-decreasing in $s\in [0, \pi]$, and for $s=0$ and $s=\pi$ it coincides with the counting functions for the Dirichlet and Neumann Laplacians on $R= \prod_{i=1}^d(0, l_i)$, respectively.

We begin by establishing the claimed comparison with the Robin counting function.
\begin{proposition}\label{prop: NRobin < NShifted}
  Fix $d\geq 2$ and $\beta \geq 3$. For any $s \in [2\arctan(3/\beta), \pi]$, $\lambda > 0$, and cuboid $R\subset \R^d$ it holds that
  \begin{equation*}
    N(-\Delta_{R}^{\beta\sqrt{\lambda}}, \lambda) \leq N_{s, R}(\lambda)\, .
  \end{equation*}
\end{proposition}

We split the proof of Proposition \ref{prop: NRobin < NShifted} into two regimes according to the size of the quantity $\min_i l_i \sqrt{\lambda}$. When $\min_i l_i \sqrt{\lambda}$ is sufficiently small, the desired estimate follows immediately from the fact that the Robin counting function vanishes.

\begin{lemma} \label{lem: lambda1_greater}
Fix $d\geq 2$ and $\beta >1$. If $R= \prod_{i=1}^d(0, l_i)$ and $\lambda \in (0, \pi^2/(2\min_i l_i)^2]$ then 
\begin{equation*}
  N(-\Delta_R^{\beta\sqrt{\lambda}}, \lambda) =0\, .
\end{equation*}
\end{lemma}

\begin{proof}
Let $k=\sqrt{\lambda_1(-\Delta_{(0, 1)}^{\beta x})}\in(0, \pi)$ and $x \in (0, \pi/2)$. It is well known that $k\in(0, \pi)$ is the unique solution of the equation 
  \begin{align*}
  k\tan(k/2)=\beta x.
\end{align*}
Using that $\tan(x/2) \leq 1$ for $x \in (0, \pi/2]$, we obtain that
\begin{equation*}
  x\tan(x/2) \leq x < \beta x = k \tan(k/2) \, .
\end{equation*}
Since $(0, \pi/2) \ni y \mapsto y \tan(y/2)$ is strictly increasing, we conclude that $k>x$ or equivalently that $\lambda_1(-\Delta_{(0, 1)}^{\beta x}) >x^2$ for $x\in (0, \pi/2)$. 
By using the product structure of $R$ and rescaling we arrive at 
\begin{equation*}
  \lambda_1(-\Delta_R^{\beta\sqrt{\lambda}}) \geq \lambda_1(-\Delta_{(0, \min_i l_i)}^{\beta \sqrt{\lambda}}) = (\min_il_i)^{-2}\lambda_1(-\Delta_{(0, 1)}^{\beta \min_i l_i\sqrt{\lambda}})>\lambda\, . 
\end{equation*}
This proves that $N(-\Delta_R^{\beta \sqrt{\lambda}}, \lambda) =0$, which completes the proof.
\end{proof}

For the regime where $\min_i l_i\sqrt{\lambda}$ is bounded from below, we can instead compare the Robin counting function directly with a suitably shifted lattice-point count.

\begin{lemma} \label{lem: sec4_Ns_estimate}
Fix $d\geq 2$, $\beta>0$, and $\Lambda > 0$. Let $R =\prod_{i=1}^d(0, l_i)$ and set 
\begin{align*}
  s_0 = 2 \arctan\biggl( \frac{1}{\beta} \biggl( 1 + \frac{\pi}{ \min_il_i \sqrt{\Lambda}} \biggr) \biggr)\, .
\end{align*}
Then, for all $\lambda \geq \Lambda$ and $s \in [s_0, \pi]$, 
\begin{align*}
    N(-\Delta_R^{\beta\sqrt{\lambda}}, \lambda) \leq N_{s, R}(\lambda)\, .
\end{align*}
\end{lemma}

\begin{proof}
  By the monotonicity of $N_{s, R}(\lambda)$ with respect to $s$ it suffices to prove the claimed inequality for $s=s_0$. Furthermore, by \eqref{eq: DN bracketing counting} it suffices to prove that $N^-_{\beta, R}(\lambda)\leq N_{s_0, R}(\lambda)$ for $\beta, \lambda$ as in the statement.
  
  Assume that $k \in \mathbb{N}^d$ satisfies 
  \begin{equation*}
    \sum_{i=1}^d\frac{1}{l_i^2} \biggl( \pi k_i - 2 \arctan\biggl( \frac{\pi k_i}{\beta l_i \sqrt{\lambda}}\biggr) \biggr)^2\leq \lambda\, .
  \end{equation*}
  Then, for each $i=1, ..., d$, 
  \begin{align*}
    \frac{\pi^2}{l_i^2} ( k_i - 1)^2 < \frac{1}{l_i^2} \biggl( \pi k_i - 2 \arctan\biggl( \frac{\pi k_i}{\beta l_i \sqrt{\lambda}}\biggr) \biggr)^2 \leq \lambda\, , 
  \end{align*}
  and so it follows that $\pi k_i \leq l_i \sqrt{\lambda} + \pi$ for each $i = 1, \ldots, d$. For $\lambda \geq \Lambda$ we conclude that
  \begin{align*}
    \sum_{i=1}^d \frac{1}{l_i^2} \biggl( \pi k_i - 2 \arctan\biggl( \frac{\pi k_i}{\beta l_i \sqrt{\lambda}}\biggr) \biggr)^2
    & \geq \sum_{i=1}^d \frac{1}{l_i^2} \biggl( \pi k_i - 2 \arctan\biggl( \frac{1}{\beta} \biggl( 1 + \frac{\pi}{ \min_jl_j\sqrt{\Lambda}} \biggr) \biggr) \biggr)^2\\
    &=Q_{s_0, R}(k)\, .
  \end{align*}
Thus, $Q_{s_0, R}(k) \leq \lambda$ and therefore 
\begin{align*}
   N_{\beta, R}^-(\lambda) \leq \# \{ k \in \mathbb{N}^d : Q_{s_0, R}(k) \leq \lambda \} = N_{s_0, R}(\lambda)\, , 
\end{align*}
which completes the proof.
\end{proof}

The proof of Proposition \ref{prop: NRobin < NShifted} is now simply a matter of suitably combining the preceding two lemmas.

\begin{proof}[Proof of Proposition \ref{prop: NRobin < NShifted}]
  Since $\beta \geq 3$ and $N_{s, R}(\lambda)\geq 0$ the claimed bound follows trivially from Lemma \ref{lem: lambda1_greater} for all $\lambda \in (0, \pi^2/(2 \min_i l_i)^2]$.

  For $\lambda \geq \pi^2/(2 \min_i l_i)^2$ we apply instead Lemma~\ref{lem: sec4_Ns_estimate} with $\Lambda= \pi^2/(2\min_il_i)^2$. This yields that
  \begin{equation*}
    N(-\Delta_R^{\beta\sqrt{\lambda}}, \lambda) \leq N_{s, R}(\lambda)
  \end{equation*}
  for all $s \in[2\arctan(3/\beta), \pi]$.
\end{proof}

In view of Proposition \ref{prop: NRobin < NShifted}, the bound in Theorem \ref{thm: sc_inequality} is a consequence of the following bound for $N_{s, R}(\lambda)$.

\begin{proposition} \label{prop: shifted_lattice_bound}
Fix $d \geq 2$. Then there exists a constant $s(d)>0$ such that for any cuboid $R\subset \R^d$ and any $s \in [0, s(d)]$, $\lambda \geq 0$ we have
\begin{equation}\label{eq: shifted lattice bound nonstrict}
  N_{s, R}(\lambda)  \leq L_{0, d}^{\rm sc} |R| \lambda^{d/2}\, .
\end{equation}
Moreover, if $s \in [0, s(d))$, then the inequality is strict for each cuboid $R$ and $\lambda >0$. If instead $s \in (s(d), \pi]$ then there exists a cuboid $R$ and a $\lambda>0$ for which the inequality is violated.
\end{proposition}

A key ingredient in our proof of Proposition \ref{prop: shifted_lattice_bound} is \cite[Lemma 2.1]{Gittins2017} which provides a parameter-dependent upper bound on the Dirichlet counting function.

\begin{lemma}[{\cite[Lemma 2.1]{Gittins2017}}] \label{lem: gittinsLarsonlemma}
Fix $d \geq 2$. There exist constants $A_d$, $B_d$, $b_0>0$ such that for any cuboid $R=\prod_{i=1}^d(0, l_i)$, the bound
\begin{align*}
  N(-\Delta_R^{\rm D} , \, \lambda) \leq L_{0, d}^{\mathrm sc} |R| \lambda^{d/2}\Bigl(1 - \frac{A_d b}{\min_i l_i \sqrt{\lambda}} + \frac{B_d b^2}{(\min_il_i)^2\lambda}\Bigr)
\end{align*}
holds for any $\lambda \geq 0$ and any $b \in [0, \, b_0]$.

\end{lemma}

We are now ready to prove Proposition \ref{prop: shifted_lattice_bound}.

\begin{proof}[Proof of Proposition \ref{prop: shifted_lattice_bound}]
Since $N_{s, R}(\lambda)$ is non-decreasing in $s$, the set of $s\in [0, \pi]$ for which \eqref{eq: shifted lattice bound nonstrict} holds for every cuboid and every $\lambda \geq 0$ is an interval. At this point, it could be that this interval is half-open, but we will later show that the interval is closed. At $s=0$ and $s=\pi$ the counting function $N_{s, R}$ coincides with that for the Dirichlet and Neumann Laplacians, respectively. Consequently, by the validity of P\'olya's conjecture for cuboids and the strictness of the corresponding inequalities (see, for instance, \cite[Lemmas 2.1 \& 2.2]{Gittins2017}), we know that this interval contains $0$ but not $\pi$. We can therefore define $s(d) \in [0, \pi)$ as the right endpoint (supremum) of this interval.

We first show that $s(d)>0$. To achieve this, it suffices to show that there exists $s>0$ for which \eqref{eq: shifted lattice bound nonstrict} holds for all $R$ and $\lambda \geq 0$. We assume without loss of generality that $R= \prod_{i=1}^d(0, l_i)$ with $l_1\leq ... \leq l_d$. If $\lambda \leq \pi^2/(4l_1^2)$ and $s\leq \pi/2$ then $N_{s, R}(\lambda)=0$, so the desired inequality is trivially true. It therefore remains to consider $\lambda \geq \pi^2/(4l_1^2)$.

Assume that $k \in \mathbb{N}^d$ satisfies $Q_{s, R}(k) \leq \lambda$. Then $\pi k_i \leq l_i \sqrt{\lambda} + s$ for any $i=1, ..., d$ and since 
\begin{align*}
  Q_{s, R}(k) = Q_{0, R}(k) - 2s \sum_{i=1}^d \frac{\pi k_i}{l_i^2 } + s^2 \sum_{i=1}^d \frac{1}{l_i^2 } \, , 
\end{align*}
we conclude that
\begin{align*}
  Q_{0, R}(k) & \leq \lambda + 2s \sum_{i=1}^d \frac{\pi k_i}{l_i^2 } - s^2 \sum_{i=1}^d \frac{1}{l_i^2 }  
  \leq \lambda\biggl(1 + \frac{2ds}{l_1 \sqrt{\lambda}} + \frac{ds^2}{l_1^2 \lambda}\biggr)\, .
\end{align*}
Consequently, 
\begin{equation*}
  N_{s, R}(\lambda) \leq N_0 \biggl(\lambda\Bigl(1 + \frac{2ds}{l_1 \sqrt{\lambda}} + \frac{ds^2}{l_1^2 \lambda}\Bigr)\biggr) = N\biggl(-\Delta_R^{\rm D}, \lambda\biggl(1 + \frac{2ds}{l_1 \sqrt{\lambda}} + \frac{ds^2}{l_1^2 \lambda}\biggr)\biggr)\, .
\end{equation*}
We aim to use Lemma~\ref{lem: gittinsLarsonlemma} to show that, for $s$ sufficiently small, the increase in spectral cut-off can be compensated for by the negative correction in this bound. 

Lemma \ref{lem: gittinsLarsonlemma} yields that
\begin{equation*}
   N\biggl(-\Delta_R^{\rm D}, \lambda\biggl(1 + \frac{2ds}{l_1 \sqrt{\lambda}} + \frac{ds^2}{l_1^2 \lambda}\biggr)\biggr) \leq L_{0, d}^{\mathrm sc} |R| \rho\biggl(\lambda\biggl(1 + \frac{2ds}{l_1 \sqrt{\lambda}} + \frac{ds^2}{l_1^2 \lambda}\biggr) \biggr)  \, 
\end{equation*}
for all $b \in [0, b_0]$ and $\lambda >0$ and
where 
\begin{equation*}
  \rho(\mu) = \mu^{d/2} \biggl( 1 - \frac{A_d b}{l_1\sqrt{\mu}} + \frac{B_d b^2}{l_1^2\mu} \biggr)\, .
\end{equation*}
It remains to choose $b$ and $s$, independently of $l_1$ and $\lambda$, so that 
\begin{equation*}
  \rho\biggl(\lambda\biggl(1 + \frac{2ds}{l_1 \sqrt{\lambda}} + \frac{ds^2}{l_1^2 \lambda}\biggr) \biggr)<\lambda^{d/2}
\end{equation*}
whenever $\lambda \geq \pi^2 / (4l_1^2)$. To this end, set $x=1/(l_1 \sqrt{\lambda})$. Then
\begin{equation*}
  \lambda\biggl(1 + \frac{2ds}{l_1 \sqrt{\lambda}} + \frac{ds^2}{l_1^2 \lambda}\biggr) = \lambda (1 +  2sdx + s^2 d x^2) = \lambda(1+\delta)
\end{equation*}
where $\delta=\delta(s, x) = 2sdx + s^2 d x^2$. We have
\begin{align*}
  \rho\biggl(\!\lambda\biggl(1 + \frac{2ds}{l_1 \sqrt{\lambda}} + \frac{ds^2}{l_1^2 \lambda}\biggr)\!\biggr) = \lambda^{d/2} \biggl( (1 + \delta )^{d/2} - bA_d x ( 1 + \delta )^{(d-1)/2} + b^2B_d x^2 (1 +\delta )^{(d-2)/2} \biggr) \, .
\end{align*}
Since $\lambda \geq \pi^2 / (4l_1^2)$, we have $0<x \leq 2/\pi$ and therefore 
\begin{align*}
  \delta \leq \frac{4sd}{\pi} + \frac{4 s^2 d}{\pi^2}\, .
\end{align*}
Hence, by assuming that $s$ is sufficiently small we can ensure that $\delta \leq 1$ for any $x \leq 2/\pi$. By Taylor's theorem, there exists a constant $C_d$ only depending on $d$ such that for any $0\leq \delta \leq 1$, we have
\begin{align*}
  (1+\delta)^{d/2} \leq 1 + \frac{d}{2} \delta + C_d \delta^2.
\end{align*}
Then, using $(1+\delta)^{(d-1)/2} \geq 1$ and $(1+\delta)^{(d-2)/2} \leq 2^{(d-2)/2}$, we obtain the upper bound
\begin{align*}
  \rho\biggl(\lambda\biggl(1 + \frac{2ds}{l_1 \sqrt{\lambda}} + \frac{ds^2}{l_1^2 \lambda}\biggr) \biggr) \leq \lambda^{d/2} \biggl( 1 + \frac{d}{2} \delta + C_d \delta^2 - bA_d x + 2^{(d-2)/2} b^2B_d x^2 \biggr) .
\end{align*}
Assuming that $s \leq 1$ the fact that $x\leq 2/\pi < 1$ implies
\begin{align*}
  \delta &= 2sdx + s^2 d x^2 \leq 4sdx\, .
\end{align*}
Consequently, 
\begin{equation*}
  \rho\biggl(\lambda\biggl(1 + \frac{2ds}{l_1 \sqrt{\lambda}} + \frac{ds^2}{l_1^2 \lambda}\biggr) \biggr) \leq \lambda^{d/2} \bigl( 1 + (2d^2 s - bA_d) x + ( 16 C_d d^2 s^2  + 2^{(d-2)/2} B_d b^2 ) x^2 \bigr) \, .
\end{equation*}

We may now choose 
\begin{equation*}
  b \leq \frac{A_d}{4} ( d^{-2} C_d A_d^2  + 2^{(d-2)/2} B_d )^{-1}
\end{equation*}
so small that the choice $s = b A_d/( 4d^2)$ satisfies the previous requirements. For any such choice we conclude that
\begin{equation*}
  \bigl( 1 + (2d^2 s - bA_d) x + ( 16 C_d d^2 s^2  + 2^{(d-2)/2} B_d b^2 ) x^2 \bigr) \leq \Bigl( 1 - \frac{A_d}{4} b x \Bigr) < 1
\end{equation*}
for any $0<x \leq 2/\pi<1$ and therefore
\begin{equation*}
  \rho\biggl(\lambda\biggl(1 + \frac{2ds}{l_1 \sqrt{\lambda}} + \frac{ds^2}{l_1^2 \lambda}\biggr) \biggr)< \lambda^{d/2}
\end{equation*}
for any $\lambda \geq \pi^2 / (4l_1^2)$. 

To show the validity of \eqref{eq: shifted lattice bound nonstrict} at $s=s(d)$ we argue as follows. Assume, for contradiction, that there exists a cuboid $R$ and $\lambda>0$ such that
\begin{equation*}
  N_{s(d), R}(\lambda) > L_{0, d}^{\rm sc}|R|\lambda^{d/2}\, .
\end{equation*}
Choose $t>1$ such that
\begin{equation*}
  N_{s(d), R}(\lambda) > L_{0, d}^{\rm sc}|tR|\lambda^{d/2}= L_{0, d}^{\rm sc}|R|t^d\lambda^{d/2}\, .
\end{equation*}
Since $Q_{s(d), tR}(k) = t^{-2}Q_{s(d), R}(k)$, every lattice point that is counted by $N_{s(d), R}(\lambda)$ satisfies $Q_{s(d), tR}(k)<\lambda$. Since there are only finitely many such lattice points and the functions $s \mapsto Q_{s, tR}(k)$ are continuous, there exists $s<s(d)$, sufficiently close to $s(d)$, such that all of these lattice points are counted also by $N_{s, tR}(\lambda)$. Thus
\begin{equation*}
  N_{s, tR}(\lambda) >L_{0, d}^{\rm sc}|tR|\lambda^{d/2}\, , 
\end{equation*}
contradicting the definition of $s(d)$. Therefore, \eqref{eq: shifted lattice bound nonstrict} holds also for $s=s(d)$.

It remains to prove that the inequality is strict when $s<s(d)$. Aiming for contradiction, suppose that equality in \eqref{eq: shifted lattice bound nonstrict} holds for some $s<s(d)$, $\lambda>0$, and $R$. Choose $s' \in (s, s(d))$. Since $Q_{s', R}(k)<Q_{s, R}(k)$ for every $k \in \N^d$, each lattice point counted by $N_{s, R}(\lambda)$ satisfies $Q_{s', R}(k)<\lambda$. Hence, for $t<1$ sufficiently close to $1$, the identity $Q_{s', tR}(k)=t^{-2}Q_{s', R}(k)$ implies that these points are still counted by $N_{s', tR}(\lambda)$. We thus conclude that
\begin{equation*}
  N_{s', tR}(\lambda) \geq N_{s, R}(\lambda) = L_{0, d}^{\rm sc}|R|\lambda^{d/2} > L_{0, d}^{\rm sc}|tR|\lambda^{d/2}\, .
\end{equation*}
Since $t<1$ and $s'<s(d)$, this contradicts the choice of $s(d)$. Therefore, we conclude that \eqref{eq: shifted lattice bound nonstrict} is strict when $s<s(d)$.
\end{proof}

\begin{proof}[Proof of Theorem \ref{thm: sc_inequality}]
  Propositions \ref{prop: NRobin < NShifted} and \ref{prop: shifted_lattice_bound} imply that the claimed inequality holds for all sufficiently large $\beta$. Moreover, by the variational principle $\beta \mapsto N(-\Delta_R^{\beta \sqrt{\lambda}}, \lambda)$ is non-increasing for every fixed $R$ and $\lambda$. Hence the set of $\beta>0$ for which the claimed inequality holds for all cuboids $R$ and $\lambda> 0$ is an interval unbounded from above. We define $\beta(d)$ to be its left endpoint (infimum). 
  
  To see that $\beta(d)>0$, fix a cuboid $R$ and choose $\lambda >0$ which is not an eigenvalue of $-\Delta_R^{\rm N}$. Then 
  $$
    \lim_{\beta \to 0} N(-\Delta_R^{\beta \sqrt{\lambda}}, \lambda) = N(-\Delta_R^{\rm N}, \lambda)\, .
  $$ 
  By \cite[Lemma 2.2]{Gittins2017}, we have $N(-\Delta_R^{\rm N}, \lambda)>L_{0, d}^{\rm sc}|R|\lambda^{d/2}$ and thus the desired inequality fails for all sufficiently small $\beta >0$.

  We next prove the validity of the bound for $\beta = \beta(d)$. Assume, for contradiction, that there are $R$ and $\lambda>0$ such that
  \begin{equation}\label{eq: assumed failure at critical}
    N(-\Delta_R^{\beta(d) \sqrt{\lambda}}, \lambda) > L_{0, d}^{\rm sc}|R|\lambda^{d/2}\, .
  \end{equation}
  Let $n = N(-\Delta_R^{\beta(d) \sqrt{\lambda}}, \lambda).$ By \eqref{eq: assumed failure at critical}, we can choose $t>1$ so close to $1$ that
  $$
     n> L_{0, d}^{\rm sc}|tR|\lambda^{d/2}\, .
  $$
  Define the natural isometry between $u \in H^1(tR)$ and $\tilde u \in H^1(R)$ by $u(x) = t^{-d/2}\tilde u(x/t)$ and note that
  \begin{align*}
    &\frac{\int_{tR}|\nabla u(x)|^2\, dx + \beta(d)\sqrt{\lambda} \int_{\partial (tR)}|u(x)|^2 \, d\Haus^{d-1}(x)}{\|u\|_{L^2(tR)}^2}\\
    &\qquad\qquad=
    \frac{\int_{R}|\nabla \tilde u(x)|^2\, dx + \beta(d)\sqrt{\lambda} \int_{\partial R}|\tilde u(x)|^2 \, d\Haus^{d-1}(x)}{\|\tilde u\|_{L^2(R)}^2}\\
    &\qquad\qquad\quad 
    +(t^{-2}-1)\frac{\int_{R}|\nabla \tilde u(x)|^2\, dx}{\|\tilde u\|_{L^2(R)}^2} + (t^{-1}-1)\frac{\beta(d)\sqrt{\lambda} \int_{\partial R}|\tilde u(x)|^2 \, d\Haus^{d-1}(x)}{\|\tilde u\|_{L^2(R)}^2}\, .
  \end{align*}
  For $t>1$ the sum of the last two terms in the above identity is negative for any non-trivial $\tilde u$. Therefore, the variational principle implies that for our chosen $t$ that $\lambda_n(-\Delta_{tR}^{\beta(d) \sqrt{\lambda}})<\lambda$. In particular, the continuity of the functions $\beta \mapsto \lambda_n(-\Delta_{tR}^{\beta\sqrt{\lambda}})$ allows us to conclude that there is a $\beta>\beta(d)$ for which we still have $\lambda_n(-\Delta_{tR}^{\beta\sqrt{\lambda}})<\lambda$. Thus, we have shown that there exists $t>1$ and $\beta >\beta(d)$ so that
  \begin{equation*}
    N(-\Delta_{tR}^{\beta \sqrt{\lambda}}, \lambda) \geq n > L_{0, d}^{\rm sc}|tR|\lambda^{d/2}\, ,
  \end{equation*}
  which contradicts the definition of $\beta(d)$, and therefore proves that the claimed inequality holds at $\beta = \beta(d)$.

  It remains to prove strictness when $\beta >\beta(d)$. Fix a cuboid $R$ and $\lambda>0$. If 
  $$
    N(-\Delta_R^{\beta(d)\sqrt{\lambda}}, \lambda)<L_{0, d}^{\rm sc}|R|\lambda^{d/2}
  $$ 
  then the monotonicity in $\beta$ implies that the inequality remains strict for all $\beta >\beta(d)$. 
  
  Suppose instead that 
  \begin{equation}\label{eq: equality case}
    N(-\Delta_R^{\beta(d)\sqrt{\lambda}}, \lambda)=L_{0, d}^{\rm sc}|R|\lambda^{d/2}\, .
  \end{equation}
  We first claim that $\lambda = \lambda_k(-\Delta_R^{\beta(d)\sqrt{\lambda}})$ with $k = L_{0, d}^{\rm sc}|R|\lambda^{d/2}\in \N$. Otherwise, by the continuity in $\mu$ of the eigenvalues of $-\Delta_R^{\beta(d)\sqrt{\mu}}$, one could choose $\mu<\lambda$ sufficiently close to $\lambda$ so that 
  $$
    N(-\Delta_R^{\beta(d)\sqrt{\mu}}, \mu) = N(-\Delta_R^{\beta(d)\sqrt{\lambda}}, \lambda)\, .
  $$ 
  Combining the P\'olya-type inequality at $\beta(d)$ with \eqref{eq: equality case} then yields 
  $$
    L_{0, d}^{\rm sc}|R|\lambda^{d/2}= N(-\Delta_R^{\beta(d)\sqrt{\lambda}}, \lambda)=N(-\Delta_R^{\beta(d)\sqrt{\mu}}, \mu) \leq L_{0, d}^{\rm sc}|R|\mu^{d/2}
  $$
  which is a contradiction as $\mu <\lambda$. This proves the claim that $\lambda = \lambda_k(-\Delta_R^{\beta(d)\sqrt{\lambda}})$ with $k = L_{0, d}^{\rm sc}|R|\lambda^{d/2}$. With this claim in hand the strict monotonicity of $\beta \mapsto \lambda_k(-\Delta_R^{\beta \sqrt{\lambda}})$ implies that $N(-\Delta_R^{\beta\sqrt{\lambda}}, \lambda) <N(-\Delta_R^{\beta(d)\sqrt{\lambda}}, \lambda) = L_{0, d}^{\rm sc}|R|\lambda^{d/2}$ if $\beta > \beta(d)$.  
\end{proof}


\def\myarXiv#1#2{\href{http://arxiv.org/abs/#1}{\texttt{arXiv:#1\, [#2]}}}

\end{document}